\documentclass[11pt]{article}

\usepackage{amsmath,amscd,amsfonts,amsthm, xypic, fullpage, amssymb, stmaryrd}
\input{xy}

\newcommand{\shrinkmargins}[1]{
  \addtolength{\textheight}{#1\topmargin}
  \addtolength{\textheight}{#1\topmargin}
  \addtolength{\textwidth}{#1\oddsidemargin}
  \addtolength{\textwidth}{#1\evensidemargin}
  \addtolength{\topmargin}{-#1\topmargin}
  \addtolength{\oddsidemargin}{-#1\oddsidemargin}
  \addtolength{\evensidemargin}{-#1\evensidemargin}
  }

\shrinkmargins{.7}

\DeclareMathOperator{\Hom}{Hom}
\DeclareMathOperator{\Ext}{Ext}

\DeclareMathOperator{\ad}{ad}

\DeclareMathOperator{\GL}{GL}

\DeclareMathOperator{\Fred}{Fred}

\DeclareMathOperator{\PU}{PU}

\DeclareMathOperator{\U}{U}

\DeclareMathOperator{\spinc}{spin_{\C}}

\newcommand{\field}[1]{\mathbb{#1}}

\newcommand{\Z}{\field{Z}}

\newcommand{\F}{\field{F}}

\newcommand{\R}{\field{R}}
\newcommand{\C}{\field{C}}
\newcommand{\Th}{\mathbf{Th}}
\newcommand{\ehat}{\hat{E}}

\newcommand{\Fhat}{\hat{F}}

\newcommand{\Lhat}{\hat{L}}
\newcommand{\Qhat}{\hat{Q}}
\newcommand{\Qbar}{\overline{Q}}

\newcommand{\ihat}{\hat{i}}

\newcommand{\hhat}{\hat{H}}
\newcommand{\vhat}{\hat{v}}
\newcommand{\what}{\hat{w}}
\newcommand{\shat}{\hat{S}}
\newcommand{\pitilde}{\tilde{\pi}}
\newcommand{\pihat}{\hat{\pi}}
\newcommand{\phat}{\hat{p}}

\newcommand{\dhat}{\hat{\Delta}}

\newcommand{\KK} {\mathcal{K}}

\newcommand{\id}{\mbox{id}}

\newcommand{\HH}{\mathcal{H}}

\newcommand{\beq}{\begin{displaymath}}
\newcommand{\eeq}{\end{displaymath}}
\newcommand{\beqn}{\begin{equation}}
\newcommand{\eeqn}{\end{equation}}

\theoremstyle{plain}
\newtheorem{thm}{Theorem}[section]
\newtheorem{prop}[thm]{Proposition}
\newtheorem{cor}[thm]{Corollary}
\newtheorem{lem}[thm]{Lemma}

\theoremstyle{definition}
\newtheorem{defn}[thm]{Definition}

\theoremstyle{remark}
\newtheorem{rem}[thm]{Remark}

\title{Topological T-duality is twisted Atiyah duality}
\author{Craig Westerland}

\begin{document}
\bibliographystyle{amsalpha}

\maketitle
\begin{abstract}

We show that the topological T-duality for circle bundles introduced in \cite{bev} can be interpreted as a form of Atiyah duality for twisted $K$-theory.

\end{abstract}

\section{Introduction}

T-duality arises in physics as an equivalence between Type IIA and Type IIB string theory on different manifolds.  This has consequences for the twisted $K$-theories of these manifolds, as they are the receptacles for certain fields in these theories.

The topological aspects of T-duality, as investigated in articles such as \cite{bev, bs, bhv, bv_loop, brs}, can be phrased in simplest form as follows.  For a $\spinc$-manifold $M$, the set of pairs $(E, H)$, consisting of a principal $S^1$-bundle $\pi: E \to M$ and a class $H \in H^3(E)$ come in dual pairs.  For such a circle bundle $E$, we write $c_1(E) \in H^2(M)$ for the first Chern class of the associated complex line bundle.  Then for such a pair $(E, H)$, the dual pair is another circle bundle $\pihat: \ehat \to M$ and class $\hhat \in H^3(\ehat)$.  These data are woven together in the following way\footnote{We follow the convention of \cite{bs} rather than \cite{bev} for the signs in these formulas; the results are true for either formulation.  We refer the reader to section \ref{model_section} for a discussion of signs in the twistings.}:
$$\begin{array}{ccc}
\pi_!(H) = -c_1(\ehat), & {\rm and} & \pihat_!(\hhat) = -c_1(E).
\end{array}$$
The remarkable fact about T-dual circle bundles is that their twisted $K$-theories are isomorphic:
\beqn K_H^*(E) \cong K^{*+1}_{\hhat}(\ehat). \label{iso_eqn} \eeqn

This isomorphism was introduced in \cite{bev} and re-examined in \cite{bs}.  In this paper, we interpret this isomorphism as a form of Atiyah duality for modules over the $K$-theory spectrum $\KK$.

Classically, Atiyah duality is a homotopy equivalence of spectra 
$$a: M^{-TM} \to F(M_+, S).$$
Here $M^{-TM}$ is the \emph{Thom spectrum} of the normal bundle of an embedding of $M$ into $\R^N$, and $F(M_+, S)$ is the \emph{Spanier-Whitehead dual} of $M$, the function spectrum of maps from $M_+$ to the sphere spectrum $S$.   This equivalence arises as the adjoint of a pairing $M^{-TM} \wedge M_+ \to S$ which realizes the intersection pairing in the stable homotopy category.
 Our aim in this paper is a proof of the following similar fact:

\begin{thm} \label{main_thm}

There is a nondegenerate pairing
$$\mu: \KK^H(E)^{-TM} \wedge_\KK \KK^{-\hhat}(\ehat) \to \KK$$
of $\KK$-modules; that is, its adjoint is an equivalence $t: \KK^H(E)^{-TM} \to F_\KK(\KK^{-\hhat}(\ehat), \KK)$.

\end{thm}

Here $\KK$ is the complex $K$-theory spectrum, and $\KK^H(E)$ is the spectrum whose homotopy groups are the twisted $K$-\emph{homology} $K^H_*(E)$.  The same geometry that defines $M^{-TM}$ allows one to define a Thom spectrum $ \KK^H(E)^{-TM}$; when $H=0$, this spectrum is equivalent to $\KK \wedge E^{-TM}$.  Further, these spectra are $\KK$-modules;  $F_\KK(\KK^{-\hhat}(\ehat), \KK) =: \KK_{-H}(\ehat)$ is the function spectrum of $\KK$-module maps from $\KK^{-\hhat}(\ehat)$ to $\KK$.  It plays the same role in the category of $\KK$-modules that the Spanier-Whitehead dual plays in the stable homotopy category, and its homotopy groups compute the the twisted $K$-cohomology $K^*_{-H}(\ehat)$.

The T-duality isomorphism (\ref{iso_eqn}) is equivalent to Theorem \ref{main_thm} via a twisted version of a somewhat standard argument using classical Atiyah and Spanier-Whitehead duality.

One aspect of this duality is its inherently ``derived" nature.  That is, the nondegeneracy of the pairing $\mu$ is rarely realized on the level of homotopy groups.  It is often the case that the twisted $K$-theory groups in question are torsion.  Since $K_*$ is, however, not torsion, it is then impossible for $\mu_*$ itself to be a nondegenerate pairing of $K_*$-modules.  Nonetheless, the adjoint of $\mu$ is an equivalence, a subtlety that is played out in the degree shift in (\ref{iso_eqn}), by way of the $\Ext$ term in the universal coefficient theorem\footnote{The reader should be reminded of the degree shifts in the torsion of the (co)homology of manifolds via Poincar\'{e} duality.}. 

\subsection{Acknowledgements}

Thanks are due to Jacobien Carstens for many helpful conversations about this material.  It is also a pleasure to thank Peter Bouwknegt and Varghese Mathai for their time and help in explaining \cite{bev}, and to beg their forgiveness for inflicting stable homotopy theory upon their unsuspecting research program.  It is worth pointing out that it is not the purpose of this paper to reprove or extend the T-duality isomorphism (\ref{iso_eqn}).  What is new here is the ``intersection theory" point of view; otherwise, this is a meditation on \cite{bev, bs}, and an advertisement for T-duality to stable homotopy theorists.

This project was begun quite some time ago at the University of Melbourne and finished at the University of Minnesota.  I would like to thank both institutions for their support, as well as acknowledging the support of the ARC (via the Future Fellowship FT100100307) and the NSF (via DMS-1406162).

\subsection{Notation and language}  

For a map $f: X \to Y$ of spaces, we use the convention that for any homology theory $h_*$, $f_*: h_*(X) \to h_*(Y)$ denotes the induced map in homology.  If $f$ admits a codimension $d$ umkehr map (e.g., a Becker-Gottlieb type transfer or a Pontrjagin-Thom collapse/intersection map), we will write the induced map in homology as $f^!: h_*(Y) \to h_{*-d}(X)$.  In cohomology, the usual induced map is denoted $f^*: h^*(Y) \to h^*(X)$, and the transfer $f_!: h^*(X) \to h^{*+d}(Y)$.

We will also write $f^!$ for umkehr map itself in the stable homotopy category; i.e., $f^!: Y \to X^\nu$, where $\nu = f^*(TY) - TX$.  Here, for a vector bundle $\xi$ over $X$, we write $X^\xi$ for the Thom space of the bundle. 

\section{Classical Poincar\'e duality}

\subsection{The intersection pairing}

Let $h$ be a multiplicative cohomology theory and let $M$ be a closed manifold which is orientable with respect to $h$.  Examine the diagram
$$\xymatrix@1{
M \times M & M \ar[l]_-{\Delta} \ar[r]^-{c} & pt,
}$$
where $\Delta$ is the diagonal, and $c$ is the constant map to a point.  Orientability ensures that $c$ admits a pushforward (or shriek) map in $h^*$; the composite
$$\xymatrix@1{
h^*(M) \otimes_{h^*} h^*(M) \ar[r]^-{\times} & h^*(M \times M) \ar[r]^-{\Delta^*} & h^*(M) \ar[r]^-{c_!} & h^{*-\dim M}(pt)
}$$
is the cup product pairing in $h$-theory.  The adjoint is a map $h^*(M) \to \Hom_{h^*}(h^{*-\dim M}(M), h^*)$.  When $h^* = H^*( \cdot; \F)$ is singular cohomology with field coefficients, the codomain of this map is $\Hom_{\F}(H^{*-\dim M}(M), \F) \cong H_{*-\dim M}(M)$, and this is the usual Poincar\'e duality isomorphism.

In $h$-homology, a similar construction may be made; the orientability of $M$ provides a pushforward map for $\Delta$ via the Pontrjagin-Thom construction:
$$\xymatrix@1{
h_*(M) \otimes_{h_*} h_*(M) \ar[r]^-{\times} & h_*(M \times M) \ar[r]^-{\Delta^!} & h_{*-\dim M}(M) \ar[r]^-{c_*} & h_{*-\dim M}(pt).
}$$
When $h_*$ is a homology theory consisting of geometric cycles, the composite of the first two maps can often be interpreted as the intersection of those cycles in $M$.

\subsection{Atiyah duality}

If $e: M \to \R^N$ is an embedding with normal bundle $\nu$, we will write $M^{-TM}$ for the Thom spectrum
$$M^{-TM} := \Sigma^{-N} M^{\nu},$$
the desuspension (in the stable homotopy category) of the Thom space of $\nu$.

The intersection product may be constructed in the category of spectra as follows: the Pontrjagin-Thom collapse map for $\Delta$ is of the form $\Delta^!: M \times M \to M^{TM}$, where $TM$ is the tangent bundle of $M$ (masquerading as the normal bundle of $\Delta$), and $M^{TM}$ is its Thom space.  Adding a factor of $\nu$ and desuspending by $N$, we obtain a composite map
$$\xymatrix@1{
M_+ \wedge M^{-TM} \ar[r]^-{\Delta^!} & M_+ \ar[r]^-{c} & S.
}$$

The adjoint map 
$$a: M^{-TM} \to F(M_+, S)$$
is an equivalence; this is \emph{Atiyah duality}.

To say that $M$ is $h$-oriented is to say that there is a Thom isomorphism in $h$-theory for $TM$.  That is, there is an equivalence $T: \Sigma^{-\dim M} h \wedge M_+ \simeq h \wedge M^{-TM}$.  Then smashing $a$ with $h$ gives
$$\xymatrix@1{
\Sigma^{-\dim M} h \wedge M_+ \ar[r]^-{T} & h \wedge M^{-TM} \ar[r]^-{1 \wedge a} & h \wedge F(M_+, S) \ar[r]^-{\simeq} & F(M_+, h).
}$$
Classical Poincar\'e duality follows from this equivalence: the homotopy groups of the domain are $h_{*+\dim M}(M)$, and those of the target are $h^{-*}(M)$.

\section{Twisted $K$-theory}

Twisted $K$-theory is an invariant associated to a topological space $X$ and a class\footnote{Really, this is only an invariant of the representative cocycle for $H$; cohomologous cocycles define isomorphic twisted K-theories, but the isomorphism itself depends upon the choice of coboundary.}  $H \in H^3(X)$.  We briefly review its construction and properties; while many of these facts are well-established in the literature, we include some equivariant homotopy-theoretic proofs in section \ref{equivariant_section}.

\subsection{The definition}

One may regard $H \in H^3(X)$ as a map $H: X \to K(\Z, 3)$.  We recall that $K(\Z, 2) = \C P^\infty$ is a summand of $\GL_1(\KK) = BU_\otimes \times \{ \pm 1 \}$; this is the inclusion of the Picard group of line bundles into the multiplicative group of virtual bundles of dimension $\pm 1$.  Here, the unit space $\GL_1(E)$ of an $E_\infty$-ring spectrum $E$ is an infinite loop space via the multiplication in $E$; this gives rise to an $E_\infty$ action of $\GL_1(E)$ upon $E$.

This map $\C P^\infty \to BU_\otimes \times \{ \pm 1\}$ deloops, and the composite
$$\xymatrix@1{X \ar[r]^-{H} & K(\Z, 3) \ar[r] & B\GL_1(\KK)}$$
defines a $\KK$-module Thom spectrum $\KK^H(X)$ over $X$; see, e.g., \cite{abg}, where this spectrum would be notated $X^H$.   Specifically, $\KK^H(X)$ may be constructed as the derived smash product\footnote{For the rest of the paper, we will tend to omit the $\Sigma^\infty$ and $L$ from the notation for this construction, writing it simply as $(\Phi_H)_+ \wedge_{K(\Z, 2)} \KK$.}
$$\KK^H(X) \simeq \Sigma^\infty (\Phi_H)_+ \wedge^L_{\Sigma^\infty K(\Z, 2)_+} \KK.$$
where $\phi: \Phi_H \to X$ is the principal $K(\Z, 2)$-bundle associated to the map $H$.

\begin{defn} The (unreduced) \emph{$H$-twisted $K$-(co)homology groups of $X$} are the homotopy groups
$$\begin{array}{ccc} K^H_*(X) := \pi_* \KK^H(X) & {\rm and} & K_H^*(X) := \pi_{-*} \KK_H(X) \end{array}$$
where we define $\KK_H(X)$ to be the function spectrum of $\KK$-module maps $\KK_H(X) := F_\KK(\KK^H(X), \KK)$.

\end{defn}

It is worth noting that when $H=0$, $\KK^H = X_+ \wedge \KK$ and $\KK_H(X) = F(X_+, \KK)$.  The homotopy groups of these spectra do indeed compute the untwisted $K$-homology and cohomology of $X$.

Notice that the analogue of the universal coefficient theorem holds by definition: for any space $X$ and twisting $H \in H^3(X)$, we will write $u$ for the equality
$$u: F_\KK(\KK^H(X), \KK) \to \KK_{H}(X).$$
Then the universal coefficient spectral sequence (e.g., from \cite{ekmm} Theorem IV.4.1) in this setting is
$$\Ext_{K^*}^{p, q}(K^H_*(X), K^*) \implies K_H^{p+q}(X).$$

Classically \cite{atiyah-segal}, one defines these groups via the description of K-theory in terms of Fredholm operators.  One model for $K(\Z,2)$ is the group $\PU(\HH)$, the projective unitary group of a Hilbert space $\HH$.   Set up in this fashion, we may regard $\Phi_H \to X$ as a principal $\PU(\HH)$-bundle.  The group $\PU(\HH)$ acts naturally on the space $\Fred(\HH)$ of Fredholm operators on $\HH$.  Defining 
$$k^0(X, H) := \Phi_H \times_{\PU(\HH)} \Fred(\HH),$$ 
then the map $k^0(X, H) \to X = \Phi_H / \PU(\HH)$ induced by $\pi$ is a fibration, with fibre $\Fred(\HH) \simeq BU \times \Z$.  

Further, since the the identity map $\id \in \Fred(\HH)$ is fixed under the action of $\PU(\HH)$, the space $\Omega \Fred(\HH)$ of loops based at the identity also admits a $\PU(\HH)$-action.  Define
$$k^1(X, H) := \Phi_H \times_{\PU(\HH)} \Omega \Fred(\HH),$$ 
This also fibres over $X$, with fibre $\Omega \Fred(\HH) \simeq \U$.  

Then twisted $K$-theory is isomorphic to the set of homotopy classes of continuous sections of these bundles:
$$K_H^i(X) \cong \Gamma[X; k^i(X, H)].$$
More generally, Bott periodicity ensures that the \emph{spaces} $\Gamma(X; k^i(X, H))$ of sections form a 2-periodic $\Omega$-spectrum; this is equivalent to $\KK_H(X)$.

\subsection{The Pontrjagin-Thom map}

For an embedding or submersion of manifolds with $\spinc$-normal bundle, we would like an umkehr map in twisted $K$-theory.  This has been accomplished in a number of papers: \cite{bev, carey_wang, fht1}.   

\begin{prop} \label{PT_map_prop}

Let $f: X \to Y$ be an embedding or submersion of compact manifolds, and write $\nu := f^*(TY) - TX$ for the normal bundle to $f$ (or negative vertical tangent bundle when $f$ is a submersion).  If $\nu$ is $\spinc$, then for each $H \in H^3(Y)$, there is a natural umkehr map
$$f^!: K^H_*(Y) \to K^{f^*H}_{*-\dim \nu}(X)$$

\end{prop}

We will give a construction of this map in section \ref{shriek_section}, using equivariant homotopy-theoretic methods.  First, however, we will present it in the stable category.  Let $p: \xi \to X$ a vector bundle; then $\phi^* \xi \to \Phi_H$ is a $K(\Z, 2)$-equivariant vector bundle.

\begin{defn}

The \emph{twisted $K$-theory Thom spectrum of $\xi$} is $\KK^H(X)^\xi := \Phi_H^{\phi^*(\xi)} \wedge_{K(\Z, 2)} \KK$.

\end{defn}

For instance, when $\xi$ is the trivial $n$-plane bundle, $\KK^H(X)^\xi \simeq \Sigma^n \KK^H(X)$.  This holds more generally when $\xi$ is $\spinc$.  Also note that this definition extends to virtual bundles by desuspension.

For an embedding or submersion of manifolds with $\spinc$-normal bundle, Proposition \ref{PT_map_prop} has a realization on the level of spectra: there is a Pontrjagin-Thom collapse map
$$f^!:  \KK^H(Y) \to \KK^{f^*H}(X)^\nu.$$
induced by the usual Pontrjagin-Thom map $Y \to X^\nu$.  We recover the umkehr map of Proposition \ref{PT_map_prop} when $\nu$ is $\spinc$ by examining the induced map in homotopy groups and using the Thom isomorphism for $\nu$. 

\subsection{Classical duality theorems in twisted $K$-theory}

Atiyah duality $a: M^{-TM} \to F(M_+, S)$ gives, after smashing with $\KK$, an equivalence
$$M^{-TM} \wedge \KK \to F(M_+, \KK).$$
This generalizes to the twisted setting in the following way:

\begin{lem} \label{atiyah_lem}

There is an equivalence $a_\KK: \KK^H(M)^{-TM} \to \KK_{-H}(M)$.

\end{lem}

We will give a proof of Lemma \ref{atiyah_lem} in section \ref{atiyah_k_section} using equivariant homotopy theory (where the sign of the twisting will also be addressed).  Taking homotopy groups and using the Thom isomorphism, this gives Poincar\'{e} duality for twisted $K$-theory:

\begin{cor} \label{PDT}

If $M$ is a $K$-oriented $d$-manifold, then there is an isomorphism
$$K^{H}_{d-*}(M) \cong K_{-H}^{*}(M).$$

\end{cor}

\section{An equivariant approach} \label{equivariant_section}

\subsection{An equivariant homotopy-theoretic model for twisted $K$-theory} \label{model_section}

Let $X$ be any space.  In principal, every homotopy theoretic question about a bundle over $X$ may be translated into a question about the action of the based loop space $\Omega X$ on the fibre of the bundle.  Using this philosophy we may give an alternate description of the twisted $K$-theory spectra. 

The loop space $\Omega X$ is, of course, the prototypical $A_\infty$ monoid.  Furthermore, the loop of the map $H: X \to K(\Z, 3)$ 
$$\rho_H: \Omega X \to K(\Z, 2)$$
is a map of $A_\infty$-spaces.  We note that one may therefore describe the $K(\Z, 2)$ bundle $\phi: \Phi_H \to X$ using this homomorphism as the Borel construction $\Phi_H \simeq E \Omega X \times_{\Omega X} K(\Z, 2)$ for this action.

Furthermore, there is an $A_\infty$ action of $\Omega X$ on the $K$-theory spectrum $\KK$ via $\rho_H$ and the $E_\infty$ map $K(\Z, 2) \to \GL_1 \KK$.  The associated homotopy fixed point and orbit spectra may be identified as the twisted $K$-(co)homology spectra:

\begin{prop} \label{borel_prop}

For the action of $\Omega X$ on $\KK$ given by $H: X \to K(\Z, 3)$, there are homotopy equivalences $\KK^H(X) \simeq \KK_{h\Omega X}$ and $\KK_{-H}(X) \simeq \KK^{h \Omega X}$.

\end{prop}

\begin{proof}

The computation
$$\KK^H(X) =  (\Phi_H)_+ \wedge_{K(\Z, 2)} \KK \\
                  \simeq (E \Omega X \times_{\Omega X} K(\Z, 2))_+ \wedge_{K(\Z, 2)} \KK \\
                  \simeq (E\Omega X)_+ \wedge_{\Omega X} \KK \\
                  = \KK_{h \Omega X} $$
gives the first result.  The second is slightly more subtle; just for this argument, write $\KK_{(H)}$ for the spectrum $\KK$, equipped with the action of $\Omega X$ given by $\rho_H$:
\begin{eqnarray*}
\KK_{-H}(X) & = & F_\KK(\KK^{-H}(X), \KK) \\
 & \simeq & F_\KK((E\Omega X)_+ \wedge_{\Omega X} \KK_{(-H)}, \KK_{(0)}) \\
 & \simeq & F_{\KK[\Omega X]}((E\Omega X)_+ \wedge \KK_{(-H)}, \KK_{(0)}) \\
 & \simeq & F_{\Omega X}((E\Omega X)_+ , \KK_{(H)}) \\
 & = & \KK_{(H)}^{h\Omega X}.
\end{eqnarray*}
The second-to-last step uses the fact that $F_\KK(\KK_{(-H)}, \KK_{(0)}) \simeq \KK_{(H)}$.  This is an equivariant refinement of the obvious equivalence $F_\KK(\KK, \KK) \simeq \KK$, arising from the equivariant pairing
$$\KK_{(H)} \wedge \KK_{(-H)} \to \KK_{(0)}$$
given by multiplication in $\KK$.

\end{proof}

Note the change in the sign of the twisting class in this result.  However, the resulting twisted $K$-groups are independent of sign: $K^*_{-H}(X) \cong K^*_{H}(X)$, as one can prove using the twisted Atiyah-Hirzebruch spectral sequence, for instance.  More directly, there is an involution $\iota$ of $\KK$ induced by the duality operation on the category of complex vector spaces.  Since there is a natural isomorphism $(V \otimes W)^* \cong V^* \otimes W^*$, $\iota$ is a ring map, and so induces an $E_\infty$ map $\iota:  \GL_1(\KK) \to  \GL_1(\KK)$.  The restriction of this map to $K(\Z, 2)$ induces negation in $\pi_2$, so the composition
$$\xymatrix@1{\Omega X \ar[r]^-{\rho_H} & K(\Z, 2) \ar[r] & \GL_1(\KK) \ar[r]^-{\iota_*} & \GL_1(\KK)
}$$
throws $\rho_H$ onto $\rho_{-H}$.  Thus $\iota$ induces equivalences of twisted K-(co)homology spectra via
$$\xymatrix@1{\KK^H(X) \simeq (E\Omega X)_+ \wedge_{\Omega X} \KK_{(H)} \ar[r]^-{1 \wedge \iota} & (E\Omega X)_+ \wedge_{\Omega X} \KK_{(-H)} \simeq \KK^{-H}(X) 
}$$

\begin{rem}

We note that if $X$ is the free loop space $X = LBG$ of the classifying space of a connected group $G$, one can identify $\Omega X \simeq LG$ as the loop group of $G$.  When $G$ is a Lie group, explicit projective unitary actions of $LG$ on a Hilbert space $\HH$ (i.e., a representative homomorphism $\rho_H: LG \to \PU(\HH)$) are given in, e.g. \cite{pressley_segal}.  If $G$ is simple and simply connected, the twisting class $H \in H^3(LBG) \cong \Z$ differs from the level $\ell$ of the representation $\rho_H$ by the dual Coxeter number of $G$.

Taking homotopy groups, we obtain
$$\pi_*(\KK^{hLG}) = K_H^*(LBG).$$
The right hand side may be identified with the completion of the level $\ell$ Verlinde algebra of $G$ at its augmentation ideal (see, e.g., \cite{fht1, kriz-wes}).  The left hand side is the $LG$ Borel-equivariant $K$-theory of a point, a stand in for the full (as yet undefined) $LG$-equivariant $K$-theory.  One expects such a theory to return the uncompleted Verlinde algebra as the value on a point.

\end{rem}

\subsection{Atiyah duality in twisted $K$-theory} \label{atiyah_k_section}

One can rephrase Atiyah duality as follows.  For any loop space $G$, there is a \emph{dualizing spectrum} 
$$S^{\ad G}:= G^{hG} = F_G(EG_+, S[G]),$$
the homotopy fixed point spectrum of the action of $G$ on the suspension spectrum of $G$ induced by left multiplication (see, e.g., \cite{klein_dual, rognes}).  Note that $S^{\ad G}$ retains a residual action of $G$ given by right multiplication.  When $G$ is the Kan loop group $G = \Omega M$ of a Poincar\'{e} duality space $M$, there is an equivariant equivalence
$$S^{\ad G} \simeq S^{-TM}$$
between the dualizing spectrum and the (desuspension of the) Spivak normal fibre of $M$, equipped with the path-lifting action.  

In \cite{klein_dual}, Klein constructed the \emph{norm map}
$$\eta: S^{\ad G} \wedge_{hG} E \to E^{hG}$$
for $G$-spectra $E$, and proved it to be an equivalence when $BG$ is finitely dominated.  In particular, for a Poincar\'e duality space $M = BG$, taking $E = S$ to be the sphere spectrum (with trivial $G$-action) gives an equivalence
$$\eta: S^{-TM} \wedge_{G} EG_+ \to F_G(EG_+, S).$$
The lefthand side is easily identified as the Atiyah dual $M^{-TM}$, and the right is the Spanier-Whitehead dual $F(M_+, S)$; thus $\eta$ realizes the Atiyah duality map $a$.

Now let $H \in H^3(M, \Z)$.  The map $\rho_H: G \to \GL_1(\KK)$ equips $G$ with a nontrivial $A_\infty$ action on $\KK$.  Taking $E = \KK$ in the definition of the norm map gives an equivalence
$$\eta: S^{\ad G} \wedge_{hG} \KK \to \KK^{hG}.$$
We have already seen that the right hand side is $\KK_{-H}(M)$, and since $S^{\ad G} \wedge_{G} EG_+ \simeq M^{-TM}$, the left hand side is equivalent to $\KK^H(M)^{-TM}$.  Thus $\eta = a_\KK$ is the Atiyah duality isomorphism of Lemma \ref{atiyah_lem}.

\subsection{The Pontrjagin-Thom map} \label{shriek_section}

Now let $X$ and $Y$ be Poincar\'e duality spaces, let $f: X \to Y$ be a continuous map, and $H \in H^3(Y)$.  We will construct an umkehr map
$$f^!: \KK^{H}(Y)^{-TY} \to \KK^{f^*H}(X)^{-TX}.$$

Write $G = \Omega X$ and $J = \Omega Y$ for the Kan loop groups of these spaces; $f$ induces a homomorphism $F := \Omega f : G \to J$.  Then as above, $\rho_H$ and $\rho_H \circ F$ define actions of $J$ and $G$ on $\KK$ (via $K(\Z, 2)$), and through this action,
$$\begin{array}{ccccc}
\KK^{H}(Y) \simeq \KK_{hJ}, & \KK_{-H}(Y) \simeq \KK^{hJ}, & \KK^{f^*H}(X) \simeq \KK_{hG}, & {\rm and} & \KK_{-f^*H}(X) \simeq \KK^{hG}.
\end{array}$$
The identity map $\KK \to \KK$ is equivariant for $F: G \to J$, so there is an induced forgetful map $f^*: \KK^{hJ} \to \KK^{hG}$.  Under the second and fourth equivalences above, this is simply the usual induced map in a (twisted) cohomology theory.

Composing $f^*$ with Atiyah duality for $X$ and $Y$, however, produces the Pontrjagin-Thom collapse (or umkehr) map $f^!$:
$$\xymatrix@1{\KK^{H}(Y)^{-TY} \ar[r]^-{a_\KK} \ar[d]_-{f^!} & \KK_{-H}(Y) \ar[d]^-{f^*} \\
                            \KK^{f^*H}(X)^{-TX} \ar[r]^-{a_\KK} & \KK_{-f^*H}(X) }
                            $$

Let $\nu = f^*(TY) - TX$; when $f$ is either an immersion or submersion, this is either the normal bundle to $f$ or the inverse of its vertical tangent bundle.  One can suspend $f^!$ by $TY$ to get
$$f^!: \KK^{H}(Y) \to \KK^{f^*H}(X)^{\nu}.$$
When $f$ is an embedding, one can identify this as the map in twisted K-theory induced by the usual Pontrjagin-Thom collapse.  When $\nu$ admits a $\spinc$ structure (e.g., when $X$ and $Y$ are $\spinc$ manifolds), its $\KK$-orientation allows us to rewrite this as $f^!: \KK^{H}(Y) \to \Sigma^{\dim Y - \dim X} \KK^{f^*H}(X)$. 

\section{T-duality}

\subsection{Correspondence diagrams}

Choose a pair of classes $F, \Fhat \in H^2(M)$.  $F$ defines a principal circle bundle $\pi: E \to M$ (e.g., as the unit circle bundle of the complex line bundle whose Chern class is $F$); similarly, $\Fhat$ defines $\pihat: \ehat \to M$. Consider the \emph{correspondence space} $E \times_M \ehat$ (an $S^1 \times \shat^1$-bundle over $M$), and the commutative diagram
\beqn \xymatrix@1{
 & E \times_M \ehat \ar[dl]_-{p} \ar[dr]^-{\phat} & \\
E \ar[dr]_-{\pi} & & \ehat \ar[dl]^-{\pihat} \\
 & M &
} \label{corr_eqn} \eeqn

Define $L$ and $\Lhat$ to be the complex line bundles on $M$ associated to the $U(1)$-bundles $E$ and $\ehat$, and let $V := L \oplus \Lhat$.  We will write $S(V)$ for the unit sphere bundle of $V$; this is an $S^3$-bundle over $M$.  Note that $F \cup \Fhat$ is the Euler class of $S(V)$.  If $F \cup \Fhat = 0$, the Gysin sequence for $r:S(V) \to M$ ensures that $S(V)$ admits a (not necessarily unique) \emph{Thom class} $\Th \in H^3(S(V))$: a class with $r_!(\Th) = 1 \in H^0(M)$.

The natural inclusions of each factor of $L \oplus \Lhat$ define embeddings
$$\begin{array}{ccc}
i: E = S(L) \hookrightarrow S(V) & {\rm and} & \ihat: \ehat = S(\Lhat) \hookrightarrow S(V)
\end{array}$$
We follow Bunke-Schick's definition of T-duality:

\begin{defn}

For classes $H \in H^3(E)$ and $\hhat \in H^3(\ehat)$,  we say that $(E, H)$ and $(\ehat, \hhat)$ are \emph{T-dual} if there exists a Thom class $\Th \in H^3(S(V))$ with $i^* \Th = H$ and $\ihat^* \Th = \hhat$. 

\end{defn}

This definition does not uniquely specify a T-dual pair $(E, H)$ and $(\ehat, \hhat)$. However, we have the following result of \cite{bs} (Corollary 2.10 and Lemmas 2.12 and 2.13; see also the discussion in section 3.1 of \cite{bev}) which characterizes these pairs:

\begin{thm} \label{bs_thm}

There exist $H \in H^3(E)$ and $\hhat \in H^3(\ehat)$ such that $(E, H)$ and $(\ehat, \hhat)$ are T-dual if and only if $F \cup \Fhat = 0$.  In this case, 
$$\begin{array}{cccc}
\pi_!(H) = -\Fhat, & \pihat_!(\hhat) = -F, & {\rm and} & p^*(H) = \phat^*(\hhat).
\end{array}$$
Finally, if such $H$ and $\hhat$ exist, then every other pair of dual cohomology classes are given by $(E, H+ \pi^* b)$ and $(\ehat, \hhat + \pihat^* b)$ for some $b \in H^3(M)$.

\end{thm}

\subsection{The T-duality isomorphism}

The definition of T-duality is a relation between $H$, $\hhat$, and $\Th$ on the level of cohomology classes.  If we let 
$$\begin{array}{cccc}
H: E \to K(\Z, 3), & \hhat: \ehat \to K(\Z, 3), & {\rm and} & \Th: S(V) \to K(\Z, 3)
\end{array}$$
be continuous maps representing these cohomology classes, then the presumption that $(E, H)$ and $(\ehat, \hhat)$ are T-dual corresponds to the \emph{existence} of homotopies $v: H \simeq \Th \circ i$ and $\vhat:\hhat \simeq \Th \circ\ihat$.  We note that the \emph{choice} of $v$ and $\vhat$ is not part of the data of the T-duality between $(E, H)$ and $(\ehat, \hhat)$, only the existence of such homotopies.

In Theorem \ref{bs_thm}, the equality $p^*(H) = \phat^*(\hhat)$ uses the definition of T-dual pairs:
$$p^*(H) = p^*i^*(\Th) = \phat^* \ihat^*(\Th) = \phat^*(\hhat)$$
as well as an explicit homotopy $h: i \circ p \simeq \ihat \circ \phat$ (see section 3.2.1 of \cite{bs}).
Refining this to the level of continuous maps, we obtain a homotopy of maps $E \times_M \ehat \to K(\Z, 3)$
$$\Lambda = (\vhat \circ \phat)^{-1} * h * (v \circ p): H \circ p \simeq \hhat \circ \phat$$
(where multiplication and inversion are done in the homotopy coordinate).  This, in turn, defines an explicit isomorphism 
\beqn \Lambda^*: K^*_{p^*H}(E \times_M \ehat) \to K^*_{\phat^* \hhat}(E \times_M \ehat) \label{u_eqn} \eeqn
Then, for any choice of $v, \vhat$, we have:

\begin{thm}{\cite{bev, bs}} The composite $\phat_! \circ \Lambda^* \circ p^*: K^*_H(E) \to K^{*-1}_{\hhat}(\ehat)$ is an isomorphism. \label{t_thm}
\end{thm}

A very brief summary of Bunke-Schick's proof of this result goes as follows.  One verifies directly that the $\phat_! \circ \Lambda^* \circ p^*$ is an isomorphism when $E$ and $\ehat$ are trivial: in that case, $\Lambda^*$ may be identified with the fibrewise multiplication by the (Poincar\'{e}) line bundle whose first Chern class is a generator of $H^2(S^1 \times \shat^1)$.  One extends this to the general setting (over any base $M$ equivalent to a finite complex) by a cellular induction and Mayer-Vietoris argument.

One can ask how different choices of the homotopies $v$ and $\vhat$ will affect this isomorphism.  The set of homotopy classes of such $v$ (resp. $\vhat$) forms a torsor for $H^2(E)$ (resp. $H^2(\ehat)$).  Changing $v$ and $\vhat$ to $w$ and $\what$, respectively, will give rise to different isomorphisms $\Lambda_v^*$ and $\Lambda_w^*$ in (\ref{u_eqn}).  The set of isomorphisms arising in this way is then a torsor for the subgroup $S := p^*H^2(E) + \phat^* H^2(\ehat) \leq H^2(E \times_M \ehat)$.  

An element of $S$ corresponds to a line bundle $\Qbar$ on $E \times_M \ehat$ which is a tensor product $\Qbar = p^*Q \otimes \phat^*\Qhat$ of line bundles pulled back from $E$ and $\ehat$.  The corresponding isomorphism $\Lambda_w^*$ is deformed from $\Lambda_v^*$ by tensoring with $Q$.  However, since $\Qbar$ is pulled back from $E$ and $\ehat$, the duality isomorphisms of Theorem \ref{t_thm} for $v$ and $w$ are related by
$$\phat_! \circ \Lambda_w^* \circ p^*(\xi) = \Qbar \otimes \phat_! \circ \Lambda_v^* \circ p^*(Q \otimes \xi).$$
This indeed shows that $\phat_! \circ \Lambda_w^* \circ p^*$ is an isomorphism if and only if $\phat_! \circ \Lambda_v^* \circ p^*$ is.

\section{T-duality as a form of Atiyah duality}

Consider the diagram
\beqn \xymatrix@1{
E \times \ehat \ar[d]_-{\pi\times \pihat} & E \times_M \ehat \ar[d]_-{\pitilde} \ar[l]_-{\dhat} \ar[dr]^-{c} \\
M \times M & M \ar[l]_-{\Delta} \ar[r]^-{c} & pt.
} \label{push_pull_eqn} \eeqn
Note that the square is Cartesian.  

The bottom row of (\ref{push_pull_eqn}) is the diagram from which the Atiyah duality isomorphism $a$ was constructed.  Passing along the top of the diagram gives a map 
$$c \circ \dhat^!:E^{-TM} \wedge \ehat_+ \to S,$$
since the normal bundle to $\dhat$ is the pullback of $TM$ to $E \times_M \ehat$.  This admits an adjoint 
$$E^{-TM} \to F(\ehat_+, S),$$
generalizing the Atiyah duality isomorphism above, though in this case the map is not an equivalence. Convolving this with the isomorphism $\Lambda^*$ of the previous section, however, will yield the T-duality isomorphism of Theorem \ref{t_thm}. 

Specifically, there is an umkehr map 
$$\dhat^!: \KK^{(H, -\hhat)}(E \times \ehat) \to \KK^{\dhat^*(H, -\hhat)}(E \times_M \ehat)^\nu,$$
where $\nu \cong TM$ is the normal bundle of $\dhat: E \times_M \ehat \to E \times \ehat$, and $(H, -\hhat)$ is the difference of the cohomology classes $H$ and $\hhat$, pulled back to the product $E\times \ehat$.

Notice that $\dhat = (p, \phat)$.  Thus $\dhat^*(H, -\hhat) = p^*(H) - \phat^*(\hhat) = 0$, so the induced $\KK$-module $\KK^{\dhat^*(H, -\hhat)}(E \times_M \ehat)$ is untwisted.  In fact, the homotopy $\Lambda$ defines a specific equivalence
$$\Lambda: \KK^{\dhat^*(H, -\hhat)}(E \times_M \ehat) \to \KK \wedge (E \times_M \ehat_+).$$
Further, $\KK^{(H, -\hhat)}(E \times \ehat) \simeq \KK^H(E) \wedge_\KK \KK^{-\hhat}(\ehat)$.  Desuspending by $TM$, we may form the composite
$$\xymatrix@1{
\mu: \KK^H(E)^{-TM} \wedge_\KK \KK^{-\hhat}(\ehat) \ar[r]^-{\dhat^!} & \KK^{\dhat^*(H, -\hhat)}(E \times_M \ehat) \ar[r]^-{\Lambda} & \KK \wedge (E \times_M \ehat_+) \ar[r]^-{1 \wedge c} & \KK.
}$$ 

\begin{defn} We will call $\mu$ the \emph{T-duality pairing}. \end{defn}

All three maps in the composite defining $\mu$ are $\KK$-module maps, so the adjoint $t: \KK^H(E)^{-TM} \to F( \KK^{-\hhat}(\ehat) , \KK)$ factors through $F_\KK(\KK^{-\hhat}(\ehat), \KK)$.  Our main result, Theorem \ref{main_thm}, is that this map
$$t: \KK^H(E)^{-TM} \to F_\KK( \KK^{-\hhat}(\ehat) , \KK)$$
is an equivalence.  That is, $\mu$ is a non-degenerate pairing of $\KK$-modules.  This follows immediately from Theorem \ref{t_thm} and the following:

\begin{prop}

There is a homotopy commutative diagram
$$\xymatrix@1{
\KK^H(E)^{-TM} \ar[d]_-{a_\KK} \ar[r]^-{t} & F_\KK( \KK^{-\hhat}(\ehat) , \KK) \ar[d]_-=^-u \\
\Sigma \KK_{H}(E) \ar[r]_{\phat_! \circ \Lambda^* \circ p^*} & \KK_{\hhat}(\ehat)
}$$
where $a_\KK$ and $u$ are the Atiyah duality and universal coefficient equivalences.  

\end{prop}

In the left column, we note that since $E$ is a principal $S^1$-bundle, there is a splitting $TE = TM \oplus 1$.   Suspending the Atiyah duality map then gives
$$a_K: \KK^H(E)^{-TM} = \Sigma \KK^H(E)^{-TE} \to \Sigma \KK_{H}(E).$$

\begin{proof} 

Passing along the lower left part of the diagram defines a map $\KK^H(E)^{-TM} \to F_\KK( \KK^{-\hhat}(\ehat) , \KK)$.  The adjoint of this map is of the form
$$\sigma: \KK^H(E)^{-TM} \wedge_\KK \KK^{-\hhat}(\ehat) \to \KK$$ 
The proposition is equivalent to the statement that this map is homotopic to $\mu$.  By definition, $\sigma$ is the suspension of the composite
$$\xymatrix@1{
\KK^H(E)^{-TE} \wedge_\KK \KK^{-\hhat}(\ehat) \ar[rr]^-{1 \wedge (p_* \circ \Lambda_* \circ \phat^!)} & & \KK^H(E)^{-TE} \wedge_\KK \Sigma^{-1} \KK^{-H}(E) \ar[r]^-{\Delta_E^!} & \Sigma^{-1} \KK \wedge E_+ \ar[r]^-{1 \wedge c}& \Sigma^{-1} \KK.
}$$ 
where $\Delta_E$ is the diagonal on $E$.  However,
$$\sigma = c_* \circ \Delta_E^! \circ [1 \wedge (p_* \circ \Lambda_* \circ \phat^!)] \simeq c_* \circ \Lambda_* \circ \Delta_{E\times_M \ehat}^! \circ (p^! \wedge \phat^!),$$
where $\Delta_{E\times_M \ehat}$ is the diagonal on $E \times_M \ehat$.

The following diagram is commutative, by inspection:
$$\xymatrix@1{
E \times_M \ehat \ar[rr]^-{ \Delta_{E\times_M \ehat}} \ar[drr]_-{\dhat} & & (E \times_M \ehat) \times (E \times_M \ehat) \ar[d]^-{p \times \phat} \\
 & & E \times \ehat
}$$
Thus $\dhat^! =  \Delta_{E\times_M \ehat}^! \circ (p^! \wedge \phat^!)$, and so
$$\sigma \simeq c_* \circ \Lambda_* \circ \dhat^! = \mu.$$

\end{proof}

\bibliography{biblio}

\end{document}